\documentclass[12pt]{article}
\renewcommand{\emph}[1]{\textit{#1}}
\usepackage{enumerate,amsmath,amsthm,latexsym,amssymb}
\usepackage{color}\usepackage{graphicx}
\usepackage{algpseudocode}

\definecolor{brown}{cmyk}{0, 0.72, 1, 0.45}
\definecolor{grey}{gray}{0.5}

\newcommand{\old}[1]{}
\newcommand{\abs}[1]{\left|#1\right|}
\newcounter{rot}

\newcommand{\ignore}[1]{}

\def\cA{{\cal A}}

\newcommand{\set}[1]{\left\{#1\right\}}
\newcommand{\ee}{\mathcal{E}}

\newcommand{\proofstart}{{\noindent \bf Proof\hspace{2em}}}
\newcommand{\proofend}{\hspace*{\fill}\mbox{$\Box$}\\ \medskip\\ \medskip}

\def\ii_(#1,#2){i_{#1}^{#2}}

\def\a{\alpha}
\def\b{\beta}
\def\d{\delta}

\def\e{\varepsilon}
\def\f{\phi}

\def\G{\Gamma}

\def\th{\theta}

\def\l{\lambda}
\def\m{\mu}
\def\n{\nu}
\def\p{\pi}

\def\s{\sigma}

\def\cE{\mathcal{E}}

\newcommand{\brac}[1]{\left( #1 \right)}

\newcommand{\expect}{\operatorname{\bf E}}
\def\E{\expect}
\renewcommand{\Pr}{\operatorname{\bf Pr}}
\newcommand\bfrac[2]{\left(\frac{#1}{#2}\right)}

\newcommand{\ep}{\varepsilon}

\newtheorem{theorem}{Theorem}
\newtheorem{conjecture}{Conjecture}
\newtheorem{lemma}{Lemma}

\newtheorem{remthm}[lemma]{Remark}

\newtheorem{observation}[lemma]{Observation}
\newcounter{thmtemp}

\newcommand{\nospace}[1]{}

\def\path{\operatorname{PATH}}

\newcommand{\al}{\alpha}
\newcommand{\chic}{\chi_{\mathrm{c}}}

\newcommand{\N}{\mathbb{N}}

\newcommand{\beq}[1]{\begin{equation}\label{#1}}
\def\eeq{\end{equation}}

\newcommand{\stm}{\setminus}
\newcommand{\dist}{\mathrm{dist}}
\newcommand{\bbb}{\mathcal{B}}
\newcommand{\st}{\mid}

\begin{document}
\title{Between 2- and 3-colorability}
\author{Alan Frieze\thanks{Research supported in part by NSF grant
    ccf1013110}, Wesley Pegden\\Department of Mathematical
  Sciences,\\Carnegie Mellon University,\\Pittsburgh PA 15213.}
\maketitle
\begin{abstract}
We consider the question of the existence of homomorphisms between $G_{n,p}$ and odd cycles when $p=c/n,\,1<c\leq 4$. We show that for any positive integer $\ell$, there exists $\e=\e(\ell)$ such that if $c=1+\e$ then w.h.p. $G_{n,p}$ has a homomorphism from $G_{n,p}$ to $C_{2\ell+1}$ so long as its odd-girth is at least $2\ell+1$. On the other hand, we show that if $c=4$ then w.h.p. there is no homomorphism from $G_{n,p}$ to $C_5$. Note that in our range of interest, $\chi(G_{n,p})=3$ w.h.p., implying that there is a homomorphism from $G_{n,p}$ to $C_3$.
\end{abstract}
\section{Introduction}
The determination of the chromatic number of $G_{n,p}$, where $p=\frac
c n$ for constant $c$, is a central topic in the theory of random
graphs.  For $0<c<1$, such graphs contain, in expectation, a bounded number of cycles,
and are almost-surely 3-colorable. The chromatic number of such
a graph may be 2 or 3 with positive probability, according as to
whether or not any odd cycles appear.

For $c\geq 1$, we find that the chromatic number $\chi(G_{n,\frac c
  n})\geq 3$ with high probability, and letting $c_k:=\sup_c
\chi(G_{n,\frac c n})\leq k$, it is known for all $k$ and $c\in
(c_k,c_{k+1})$ that $\chi(G_{n,\frac c n})\in \{k,k+1\}$, see
{\L}uczak \cite{Lu} and Achlioptas and Naor \cite{AN}; for $k>2$, the
chromatic number may well be concentrated on the single value $k$, see
Friedgut \cite{Fried} and Achlioptas and Friedgut \cite{AcFr}.

In this paper, we consider finer notions of colorability for the
graphs $G_{n,\frac c n}$ for $c\in (1,c_3)$, by considering
homomorphisms from $G_{n,\frac c n}$ to odd cycles $C_{2\ell+1}$. A
homomorphism from a graph $G$ to $C_{2\ell+1}$ implies a homomorphism to $C_{2k+1}$ for $k<\ell$.
As the 3-colorability of a graph $G$ corresponds to the existence of a
homomorphism from $G$ to $K_3$, the existence of a homomorphism to
$C_{2\ell+1}$ implies 3-colorability. Thus considering homomorphisms
to odd cycles $C_{2\ell+1}$ gives a hierarchy of 3-colorable graphs
amenable to increasingly stronger constraint satisfaction
problems. Note that a fixed graph having a homomorphism to any odd-cycle is bipartite.

Our main result is the following:
\begin{theorem}
\label{t.hom}
For any $\ell>1$, there is an $\ep>0$ such that with high probability, $G_{n,\frac{1+\ep}{n}}$ either has odd-girth $<2\ell+1$ or has a homomorphism to $C_{2\ell+1}$.  
\end{theorem}

Conversely, we expect the following:
\begin{conjecture}
\label{c.nohom}
For any $c>1$, there is an $\ell_c$ such that with high probability, there is no homomorphism from $G_{n,\frac c n}$ to $C_{2\ell+1}$ for $\ell\geq \ell_c$.  
\end{conjecture}

As $c_3$ is known to be at least $4.03$, the following confirms Conjecture \ref{c.nohom} for a significant portion of the interval $(1,c_3)$.
\begin{theorem}
\label{t.nohom}
For any $c>2.774$, there is an $\ell_c$ such that with high probability, there is no homomorphism from $G_{n,\frac c n}$ to to $C_{2\ell+1}$ for $\ell\geq \ell_c$.
\end{theorem}
We also have that $\ell_4=2$:
\begin{theorem}
\label{t.l4}
With high probability, $G_{n,\frac{4}{n}}$ has no homomorphism to $C_5$.
\end{theorem}
Note that as $c_3>4.03>4$, we see that there are triangle-free
3-colorable random graphs without homomorphisms to $C_5$.  Our proof
of Theorem \ref{t.l4} involves computer assisted numerical
computations.  The same calculations which rigorously demonstrate that $\ell_4=2$ suggest actually that $\ell_{3.75}=2$ as well.

Our results can be reformulated in terms of the \emph{circular chromatic number} of a random graph.  Recall that the circular chromatic number $\chic(G)$ of $G$ is the infimum $r$ of circumferences of circles $C$ for which there is an assignment of open unit intervals of $C$ to the vertices of $G$ such that adjacent vertices are assigned disjoint intervals.  (Note that if circles $C$ of circumference $r$ were replaced in this definition with line segments $S$ of length $r$, then this would give the ordinary chromatic number $\chi(G)$.) It is known that $\chi(G)-1< \chic(G)\leq \chi(G)$, that $\chic(G)$ is always rational, and moreover, that $\chic(G)\leq \frac{p}{q}$ if and only if $G$ has a homomorphism to the circulant graph $C_{p,q}$ with vertex set $\{0,1,\dots,q-1\}$, with $v\sim u$ whenever $\dist(v,u):=\min\{\abs{v-u},v+q-u,u+q-v\}\geq q$.  (See \cite{zhu}.)  Since $C_{2\ell+1,\ell}$ is the odd cycle $C_{2\ell+1}$ our results can be restated as follows:
\begin{theorem}
In the following, inequalities for the circular chromatic number hold with high probability.
\begin{enumerate}
\item For any $\delta>0$, there is an $\ep>0$ such that, $G=G_{n,\frac{1+\ep}{n}}$ has $\chic(G)\leq2+\delta$ unless it has odd girth $\leq \frac{2}{\delta}$.
\item For any $c>2.774$, there exists $r>2$ such that $\chic(G_{n,\frac c n})>r$.
\item $2.5\leq \chic(G_{n,\frac 4 n})<3$.
\end{enumerate}
\end{theorem}

Note that for any $c$ and $\ell>1$, there is positive probability that
$G_{n,\frac c n}$ has odd girth $<2\ell+1$, and a positive probability
that it does not.  In particular, as the probability that $G_{n,\tfrac{c}{n}}$
has small odd-girth can be computed precisely, Theorem \ref{t.hom}
gives an exact probability in $(0,1)$ that $G_{n,\frac{1+\ep}{n}}$ has
a homomorphism to $C_{2\ell+1}$. Indeed, Theorem \ref{t.hom} implies
that if $c=1+\e$ and $\e$ is sufficiently small relative to $\ell$, then
\beq{eq1}
\lim_{n\to\infty}\Pr(\chic(G_{n,\frac{c}{n}})\in (2+\tfrac{1}{\ell+1}, 2+\tfrac 1 \ell])=e^{-\f_{\ell}(c)}-e^{-\f_{\ell+1}(c)},
\eeq
where 
$$\f_\ell(c)=\sum_{i=1}^{\ell-1} \frac{c^{2i+1}}{2(2i+1)}.$$

We close with two more conjectures.  The first concerns a sort of pseudo-threshold for having a homomorphism to $C_{2\ell+1}$:
\begin{conjecture}
\label{c.sharp}
For any $\ell$, there is a $c_\ell>1$ such that $G_{n,\frac c n}$ has no homomorphism to $C_{2\ell+1}$ for $c>c_\ell$, and has either odd-girth $<2\ell+1$ or has a homomorphism to $C_{2\ell+1}$ for $c<c_\ell$.
\end{conjecture}
The second asserts that the circular chromatic numbers of random graphs should be dense.
\begin{conjecture}
There are no real numbers $2\leq a<b$ with the property that for any value of $c$, $\Pr(\chic(G_{n,\tfrac c n})\in (a,b))\to 0$.
\end{conjecture}
Note that our Theorem \ref{t.hom} confirms this conjecture for the case $a=2$.
\section{Structure of the paper}
We prove Theorem \ref{t.hom} in Section \ref{secfh}. We first prove
some structural lemmas and then we show, given the properties in these
lemmas, that we can algorithmically find a homomorphism. We prove Theorem
\ref{t.nohom} in Section \ref{ah} by the use of a simple first moment
argument. We prove Theorem \ref{t.l4} in Section \ref{ah5}. This is
again a first moment calculation, but it has required numerical
assistance in its proof.
\section{Finding homomorphisms}\label{secfh}
\begin{lemma}
\label{l.short}
If $\alpha<1/10$ and $c$ is a positive constant where 
$$c<c_0=\exp\set{\frac{1-6\a}{3\a}}$$ 
then
w.h.p. any two cycles of length less than $\alpha \log n$ in $G_{n,p}$,
$p=\frac{c}{n}$, are at distance more than $\alpha \log n$.
\end{lemma}
\proofstart
If there are two cycles contradicting the above claim, then there exists a set $S$ of size $s\leq 3\a\log n$ that contains at least $s+1$ edges. The expected number of such sets can be bounded as follows:
\begin{align*}
\sum_{s=4}^{3\a\log n}\binom{n}{s}\binom{\binom{s}{2}}{s+1}\bfrac{c}{n}^{s+1}&\leq 
\sum_{s=4}^{3\a\log n}\bfrac{ne}{s}^s\bfrac{se}{2}^{s+1}\bfrac{c}{n}^{s+1}\\
&\leq \frac{3c\a\log n}{n}\sum_{s=4}^{3\a\log n}\bfrac{ce^2}{2}^s\\
&<\frac{(ce^2)^{3\a\log n}\log n}{n}\\
&=o(1).
\end{align*}
\proofend

Our next lemma is concerned with cycles in $K_2$ which is the
{\em 2-core} of $G_{n,p}$. The 2-core of a graph is the graph induced by the edges that
are in at least one cycle. When $c>1$, the 2-core consists of a linear
size sub-graph together with a few vertex disjoint cycles. By
few we mean that in expectation, there are $O(1)$ vertices on these cycles.

Let $0<x<1$ be such that $xe^{-x}=ce^{-c}$. Then w.h.p. $K_2$ has 
$$\n\sim (1-x)\brac{1-\frac{x}{c}}n\text{ vertices and }\m\sim\brac{1-\frac{x}{c}}^2\frac{cn}{2} \text{ edges}.$$
(See for example Pittel \cite{Pit}).

If $c=1+\e$ for $\e$ small and positive then $x=1-\eta$ where
$\eta=\e+a_1\e^2,\,|a_1|\leq 2$ for $\e< 1/10$.

The degree sequence of $K_2$ can be generated as follows, see for example Aronson, Frieze and Pittel \cite{AFP}: Let $\l$ be
the solution to
$$\frac{\l(e^\l-1)}{e^\l-1-\l}=\frac{2\m}{\n}\sim\frac{c-x}{1-x}=\frac{2+a_1\e}{1+a_1\e}.$$
We deduce from this that
$$\l\leq 4|a_1|\e\leq 8\e.$$
We generate the degrees $d(1),d(2),\ldots,d(\n)$ as independent copies
of the random variable $Z$ where for $d\geq 2$,
$$\Pr(Z=d)=\frac{\l^d}{d!(e^\l-1-\l)}.$$
We condition that the sum $D_1=d(1)+d(2)+\cdots+d(n)=2\m$. We let 
\begin{align*}
\th_k&=\frac{\Pr(d(i)=d_i,i=1,2,\ldots,k\mid
  D_1=2\m)}{\Pr(d(i)=d_i,i=1,2,\ldots,k)}\\
&=\frac{\Pr(d(k+1)+\cdots+d(n)=2\m-(d_1+\cdots+d_k)}{\Pr(d(1)+\cdots+d(n)=2\m)}.
\end{align*}
It is shown in \cite{AFP} that if $Z_1,Z_2,\ldots,Z_N$ are independent
copies of $Z$ then 

\beq{Z=}
\Pr(Z_1+\cdots+Z_N=N\E(Z)-t)=\frac{1}{\s\sqrt{2\p
    N}}\brac{1+O\bfrac{t^2+1}{N\s^2}}
\eeq
where $\s^2=\Theta(1)$ is the variance of $Z$.

We observe next that the maximum degree in $G_{n,p}$ and hence in
$K_2$ is q.s.\footnote{A sequence of events $\cE_n$ is said to occur
  {\em quite surely} q.s. if $\Pr(\neg\cE_n)=O(n^{-C})$ for any
  constant $C>0$.} at most $\log n$. It follows from this and
\eqref{Z=} that 
$$\th_k=1+o(1)\text{ for }k\leq \log^2n\text{ and
}\th_k=O(n^{1/2})\text{ in general}.$$
\begin{lemma}
\label{l.long}
For any $\alpha,\beta$, there exists $c_0>1$ such that w.h.p. any cycle of length greater than $\alpha\log n$ in the 2-core of $G_{n,p}$, $p=\frac{c}{n}$, $1<c<c_0$, has at most $\beta \log n$ vertices of degree $\geq 3$.
\end{lemma}
\proofstart
Suppose that
$$e^{1+8\e}\bfrac{8\e e}{\b}^\b<1. $$
We will show then that w.h.p. the $K_2$ does not contain a cycle $C$ where (i) $|C|\geq \a\log n$ and (ii) $C$ contains $\b|C|$ vertices of degree greater than two.

We can bound the probability of the existence
of a ``bad'' cycle $C$ as follows: In the following display we choose the vertices of  our cycle in $\binom{\n}{k}$ ways and then arrange these vertices in a cycle $C$ in $(k-1)!/2$ ways. Then we choose $\b k$ vertices to have degree at least three. We then sum over possible degree sequences for the vertices in $C$. This explains the factor $\th_k\prod_{i=1}^k\frac{\l^{d_i}}{d_i!(e^\l-1-\l)}$. We now resort to using the configuration model of Bollob\'as \cite{Boll}. This would explain the product $\prod_{i=1}^k\frac{d_i(d_i-1)}{2\m-2i+1}$. We use the denominator $2\m-k$ to simplify the calculation. The configuration model computation will inflate our estimate by a constant factor that we hide with the notation $\leq_b$. We write $A\leq_b B$ for $A=O(B)$ when $O(B)$ is ``ugly looking''. 
\begin{align*}
\Pr(\exists C)&\leq_b \sum_{k=\a\log n}^\n\binom{\n}{k}\frac{(k-1)!}{2}\binom{k}{\b
  k}\th_k\sum_{\substack{d_1,\ldots,d_{\b k}\geq3\\d_{\b k+1},\ldots,d_k\geq
    2}}\prod_{i=1}^k\brac{\frac{\l^{d_i}}{d_i!(e^\l-1-\l)}\cdot\frac{d_i(d_i-1)}{2\m-2k}}\\
&\leq \sum_{k=\a\log n}^\n\frac{1}{2k}\bfrac{\n}{(2\m-2k)(e^\l-1-\l)}^k\l^{2k}\binom{k}{\b
  k}\th_k\sum_{\substack{d_1,\ldots,d_{\b k}\geq3\\d_{\b
      k+1},\ldots,d_k\geq
    2}}\prod_{i=1}^k\frac{1}{(d_i-2)!}\\
&\leq\sum_{k=\a\log n}^\n\frac{e^{k^2/\m}}{2k}\bfrac{\n}{2\m(e^\l-1-\l)}^k\l^{2k}\binom{k}{\b
  k}\th_k(e^\l-1)^{\b k}e^{(1-\b)k\l}\\
&=\sum_{k=\a\log n}^\n\frac{e^{k^2/\m}}{2k}\bfrac{\l}{e^\l-1}^k\binom{k}{\b
  k}\th_k(e^\l-1)^{\b k}e^{(1-\b)k\l}\\
&\leq\sum_{k=\a\log n}^\n\frac{\th_k}{2k}\brac{e^{k/\m}\cdot\frac{\l}{(e^\l-1)^{1-\b}}\cdot\bfrac{e}{\b}^\b\cdot
  e^{(1-\b)\l}}^k\\
&\leq \sum_{k=\a\log n}^\n\frac{\th_k}{2k}\brac{e\cdot\l^\b\cdot\bfrac{e}{\b}^\b\cdot
  e^\l}^k\\
&=o(1).
\end{align*}
\proofend

\begin{lemma}
\label{l.decomp}
For any $\alpha$ and any $k\in \N$, there exists $\ep_0>0$ such that w.h.p. we can decompose the edges of the $G=G_{n,p}$, $p=\frac{1+\ep}{n}$, $0<\ep<\ep_0$, as $F\cup M$, where $F$ is a forest, and where the distance in $F$ between any two edges in $M$ is at least $k$.
\end{lemma}
\newcommand{\ppp}{\mathcal{P}}
\proofstart
By choosing $\beta<\frac 1 {2k}$ in Lemma \ref{l.long} we can find, in every cycle of length $>\alpha \log n$ of the $2$-core $K_2$ of $G$ (which includes all cycles of $G$), a path of length at least $2k+1$ whose interior vertices are all of degree 2.  We can thus choose in each cycle of $K_2$ of length $>\alpha \log n$ such a path of maximum length, and let $\ppp$ denote the set of such paths.  (Note that, in general, there will be fewer paths in $\ppp$ than long cycles in $K_2$ due to duplicates, but that the elements of $\ppp$ are nevertheless disjoint paths in $K_2$.)  We now choose from each path in $\ppp$ an edge from the center of the path to give a set $M_1$.  Note that the set of cycles in $G\stm M_1$ is the same as the set of cycles in $G\stm \bigcup_{P\in \ppp} P$. (In particular, the only cycles which remain have length $\leq \alpha \log n$ and are at distance $\geq k$ from $M$.)  Thus, letting $M_2$ consist of one edge from each cycle of $G\stm M_1$, Lemma \ref{l.short} implies that $M=M_1\cup M_2$ is as desired.
\proofend

\begin{proof}[Proof of Theorem \ref{t.hom}]
Our goal in this section is to give a $C_{2\ell+1}$-coloring of $G=G_{n,\frac{1+\ep}{n}}$ for $\ep>0$ sufficiently small.  By this we will mean an assignment $c:V(G)\to \{0,1,\dots,2\ell\}$ such that $x\sim y$ in $G$ implies that $c(x)\sim c(y)$ as vertices of $C_{2\ell+1}$; that is, that $x=y\pm 1\pmod {2\ell+1}$.

Consider a decomposition of $G$ as $F\cup M$ as given by Lemma \ref{l.decomp}, with $k=4\ell-2$.

We begin by 2-coloring $F$. Let $c_F:V\to \{0,1\}$ be such a coloring.  Our goal will be to modify this coloring to give a good $C_{2\ell+1}$ coloring of $S$.

Let $\bbb$ be the set of edges $xy\in M$ for which $c_F(x)=c_F(y)$, and let $B$ be a set of distinct representatives for $\bbb$, and for $i=0,1$, let $B^i=\{v\in B \st c_F(v)=i\}$.

We now define a new $C_{2\ell+1}$ coloring $c:V\to \{0,1,\ldots,2\ell\}$, by
\begin{equation}\label{qw}
c(v)=
\begin{cases}
c_F(v) & \mbox{if }\dist_F(v,B)\geq 2\ell-1\\
c_F(x)-(-1)^j(\dist_F(x,v)+1) & {\mbox{if } \exists x\in B^j \mbox{ s.t. } \dist(x,v)_F<2\ell-1.}\\
\end{cases}
\end{equation}
(Color addition and subtraction are computed modulo $2\ell+1$.)

Since edges in $M$ are separated by distances $\geq 4\ell-2$, this coloring is well-defined (i.e., there is at most one choice for $x$).  Moreover, $c$ is certainly a good $C_{2\ell+1}$-coloring of $F$.  Thus if $c$ is a not a good $C_{2\ell+1}$-coloring of $S$, it is bad along some edge $xy\in M.$  But if such an edge was already properly colored in the 2-coloring $c_F$, it is still properly colored by $c$, since it has distance $\geq 4\ell-2\geq 2\ell-1$ from other edges in $M$.  On the other hand, if previously we had $c_F(x)=c_F(y)=i$, and WLOG $x\in B^i$, then the definition of $c(v)$ gives that we now have that $c(x)\in\set{i-1,i+1}$ (modulo $2\ell-1$).  Thus if $c$ is not a good $C_{2\ell+1}$-coloring of $S$, then there is an edge $xy\in M$ such that $x\in B^i$ and $y$'s color also changes in the coloring $c$; but by the distance between edges in $M$, this can only happen if $x$ and $y$ are at $F$-distance $<2\ell-1$.  Note also that $c_F(x)=c_F(y)$ implies that $\dist_F(x,y)$ is even.  Thus in this case, $F\cup \{xy\}$ contains an odd cycle of length $\leq 2\ell-1$, and so $G$ has odd girth $<2\ell+1$, as desired.
\end{proof}

\section{Avoiding homomorphisms to long odd cycles}\label{ah}
For large $\ell$, one can prove the non-existence of homomorphisms to $C_{2\ell+1}$ using the following simple observation:
\begin{observation}
If $G$ has a homomorphism to $C_{2\ell+1}$, then $G$ has an induced bipartite subgraph with at least $\frac{2\ell}{2\ell+1} |V(G)|$ vertices.
\end{observation}
\begin{proof}Delete the smallest color class.\end{proof}
\begin{proof}[Proof of Theorem \ref{t.nohom}]
The probability that $G_{n,\frac c n}$ has an induced bipartite subgraph on $\beta n$ vertices is at most
\begin{equation}\label{e.bipartite}
\binom{n}{\beta n}2^{\beta n}\brac{1-\frac{c}{n}}^{\beta^2 n^2/4}<\left(\frac{2^\beta e^{-c\beta^2/4}}{\beta^\beta(1-\beta)^{1-\beta}}\right)^n
\end{equation}
The expression inside the parentheses is unimodal in $\beta$ for fixed $c$, and, for $c>2.774$, is less than 1 for $\beta>.999971$.  In particular, for $c>2.774$, $G_{n,\frac c n}$ has no homomorphism to $C_{2\ell+1}$ for $2\ell+1\geq 1,427,583$.  
\end{proof}
\section{Avoiding homomorphisms to $C_5$}\label{ah5}
A homomorphism of $G=G_{n,p},p=\frac c n$ into $C_5$ induces a partition
of $[n]$ into sets $V_i,i=0,1,\ldots,4$. This partition can be assumed
to have the following properties:
\begin{enumerate}[{\bf P1}]
\item The sets $V_i,i=0,1,\ldots,4$ are all independent sets.
\item There are no edges between $V_i$ and $V_{i+2}\cup V_{i-2}$. Here
  addition and subtraction in an index are taken to be modulo 5.
\item Every $v\in V_i,i=1,2,3,4$ has a neighbor in $V_{i-1}$. 
\item Every $v\in V_2$ has a neighbor in $V_3$.
\end{enumerate}
Hatami \cite{Hat}, Lemma 2.1 shows that we can assume {\bf
  P1,P2,P3}. Given {\bf  P1,P2,P3}, if $v\in V_2$ has no neighbors in $V_3$ then we can move $v$ from
  from $V_2$ to $V_0$ and still have a homomorphism. Furthermore, this
  move does not upset {\bf P1,P2,P3}.
 
We let $|V_i|=n_i$ for $i=0,1,\ldots,4$. For a fixed
partition we then have
\beq{P12}
\Pr({\bf P1\wedge P2})=(1-p)^S\text{ where }S=\binom{n}{2}-\sum_{i=0}^4n_in_{i+1}.
\eeq
\beq{P3}
\Pr({\bf P3\mid P1\wedge P2})=\prod_{i=1}^4(1-(1-p)^{n_{i-1}})^{n_i}.
\eeq
\beq{P4}
\Pr({\bf P4\mid P1\wedge P2\wedge P3})\leq \brac{1-\brac{1-\frac{1}{n_2}}^{n_3}(1-p)^{n_3}}^{n_2}
\eeq
Equations \eqref{P12} and \eqref{P3} are self evident, but we need to
justify \eqref{P4}. Consider the bipartite subgraph $\G$ of $G_{n,p}$ induced by
$V_2\cup V_3$. {\bf P3} tells us that each $v\in V_3$ has a neighbor
  in $V_2$. Denote this event by $\cA$. Suppose now that we choose a random mapping $\f$ from
  $V_3$ to $V_2$. We then create a bipartite graph $\G'$ with edge set
  $E_1\cup E_2$. Here $E_1=\set{xy:x\in V_3,y=\f(x)}$ and $E_2$ is
  obtained by independently including each of the $n_2n_3$ possible edges between
  $V_2$ and $V_3$ with probability $p$. We now claim that we can
  couple $\G,\G'$ so that $\G\subseteq \G'$.

Event $\cA$ can be construed as follows: A vertex in $v\in V_3$
chooses $B_v$ neighbors in $V_2$ where $B_v$ is distributed as a
binomial $Bin(n_2,p)$, conditioned to be at least one. The neighbors
of $v$ in $V_2$ will then be a random $B_v$ subset of $V_2$. We only
have to prove then that if $v$ chooses $B_v'$ random neighbors in
$\G'$ then $B_v'$ stochastically dominates $B_v$. But $B_v'$ is one
plus $Bin(n_2-1,p)$ and domination is easy to confirm. We have $n_2-1$
instead of $n_2$, since we do not wish to count the edge $v$ to
$\f(v)$ twice.

We now write $n_i=\a_in$ for $i=0,\ldots,4$. We are particularly
interested in the case where $c=4$. Now \eqref{e.bipartite} implies
that $G_{n,\frac 4 n}$ has no induced bipartite subgraph of size
$\beta n$ for $\beta >0.94$. Thus we may assume that
$\alpha_i\geq0.06$ for $i=0,\ldots,4$. In which case we can write
\begin{multline*}
\Pr({\bf P1\wedge P2\wedge P3\wedge P4})\leq e^{o(n)}\times
\exp\set{-c\brac{\frac 1 2-\sum_{i=0}^4\a_i\a_{i+1}}n}\times
\brac{\prod_{i=1}^4(1-e^{-c\a_{i-1}})^{\a_i}}^n\times\\
(1-e^{-\a_3/\a_2}e^{-c\a_3})^{\a_2n}.
\end{multline*}
The number of choices for $V_0,\ldots,V_4$ with these sizes is
$$\binom{n}{n_0,n_1,n_2,n_3,n_4}=e^{o(n)}\times
\bfrac{1}{\prod_{i=0}^4\a_i^{\a_i}}^n\leq 5^n.$$
Putting $\al_4=1-\al_0-\al_1-\al_2-\al_3$ and
\begin{multline*}
b=b(c,\al_0,\al_1,\al_2,\al_3)=
\frac{1}{ {\al_0}^{\al_0} {\al_1}^{\al_1} {\al_2}^{\al_2} {\al_3}^{\al_3} {\al_4}^{\al_4}}\\
e^{c(\al_0\al_4-\frac 1 2)}  (e^{c \al_0}-1)^{\al_1} (e^{c \al_1}-1)^{\al_2} (e^{c \al_2}-1)^{\al_3} (e^{c \al_3}-1)^{\al_4} (1-e^{-\al_3/\al_2}e^{-c \al_3})^{\al_2},
\end{multline*}
we see that since there are $O(n^4)$ choices for $n_0,\ldots,n_4$ we have
\beq{bound}
\Pr(\exists \text{ a homomorphism from $G_{n,\frac4n}$ to $C_5$})\leq
  e^{o(n)}\brac{\max_{\substack{\a_0+\cdots+\a_3\leq
      0.94\\\a_0,\ldots,\a_3\geq
      0.06}}b(4,\al_0,\al_1,\al_2,\al_3)}^n.
\eeq
In the next section, we describe a numerical procedure for verifying
that the maximum in \eqref{bound} is less than 1. This will
complete the proof of Theorem \ref{t.l4}.

\section{Bounding the function.}
Our aim now is to bound the partial derivatives of
$b(4.0,\al_0,\al_1,\al_2,\al_3)$, to translate numerical computations
of the function on a grid to a rigorous upper bound.  

Before doing this we verify that w.h.p. $G_{n,p=\frac4 n}$ has no
independent set $S$ of size $s=3n/5$ or more. 
Indeed, 
$$\Pr(\exists S)\leq 2^n(1-p)^{\binom{s}{2}}\leq 2^ne^{-18n/25}e^{12/5}=o(1).$$
In the calculations below we will make use of the following bounds:
They assume that $0.06\leq \a_i\leq 0.6$ for $i\geq 0$.
$$\log(\a_i)>-2.82;\quad-1.31<\log(e^{4\a_i}-1)<2.31;\quad\frac{e^{4\a_i}}{e^{4\a_i}-1}<4.69$$
$$\frac{1}{e^{4\a_i}-1}<3.69;\quad\log(e^{\al_3/\al_2+4\al_3}-1)>-0.91;\quad\frac{1+4\al_2}{e^{\al_3/\al_2}e^{4\al_3}-1}<8.40.$$
We now use these estimates to bound the absolute values of the
$\frac{1}{b}\cdot\frac {\partial b}{\partial \al_i}$.
Our target value for these is 30. We will be well within these bounds
except for $i=2$

Taking logarithms to differentiate with respect to $\al_0$, we find 
\begin{multline}
\frac {\partial b}{\partial \al_0}=b(c,\al_0,\al_1,\al_2,\al_3)\times \\
\left(c\left(-\al_0+\al_1+\frac{\al_1}{e^{\al_0 c}-1}+\al_4\right) -\log (\al_0)+\log (\al_4)-\log(e^{\al_3 c}-1)\right).
\end{multline}
In particular, for $c=4$,
\begin{align*}
\frac{1}{b}\cdot\frac {\partial b}{\partial \al_0}&\geq-4\a_0+\log(\a_4)-\log(e^{4\a_3}-1)>-2.4-2.82-2.31,\\
\frac{1}{b}\cdot\frac {\partial b}{\partial
  \al_0}&\leq4\brac{\a_1+\frac{\al_1}{e^{\al_0
      c}-1}+\al_4}-\log(\a_0)-\log(e^{4\a_3}-1)<4\times 4.69+2.82+1.31.
\end{align*}
Similarly, we find
\begin{multline}
\frac {\partial b}{\partial \al_1}=b(c,\al_0,\al_1,\al_2,\al_3)\times\\
\left(c\left(-\al_0+\al_2+\frac{\al_2}{e^{\al_1 c}-1}\right) -\log (\al_1)+\log (\al_4)+\log\bfrac{e^{\al_0 c}-1}{e^{\al_3 c}-1}\right),
\end{multline}
and so for $c=4$,
\begin{align*}
\frac{1}{b}\cdot\frac {\partial b}{\partial \al_1}&\geq-4\a_0+\log(\a_4)+\log(e^{4\a_0}-1)-\log(e^{4\a_3}-1)>-2.4-2.82-3.62,\\
\frac{1}{b}\cdot\frac {\partial b}{\partial \al_1}&\leq4\brac{\a_2+\frac{\al_2}{e^{4\al_1}-1}}-\log(\a_1)-\log(e^{4\a_3}-1)<2.4\times 4.69+2.82+1.31.
\end{align*}
We next find that 
\begin{multline}
\frac {\partial b}{\partial \al_2}=b(c,\al_0,\al_1,\al_2,\al_3)\times\\
c\left(-\al_0+\al_3+\frac{\al_3}{e^{\al_2 c}-1}\right)-\frac{\al_3/\al_2}{e^{\al_3/\al_2+c\al_3}-1}+\\\log\al_4-\log \al_2 +\log(e^{\al_1 c}-1)-\log(e^{\al_3 c}-1)-\frac {\al_3}{\al_2}-c\al_3-\log(e^{\al_3/\al_2+c \al_3}-1);
\old{=
\log\left(\frac{\al_4\al_1 }{\al_2 \al_3^2 c + \al_3^2}\right)-\frac{\al_3/\al_2}{e^{\al_3/\al_2+c\al_3}-1}
+\ee(c(K+1))}
\end{multline}
and so for $c=4$,
\begin{align*}
\frac{1}{b}\cdot\frac {\partial b}{\partial \al_2}\geq&-4\a_0-\frac{\al_3}{\al_2}\frac{e^{\al_3/\al_2+c\al_3}}{e^{\al_3/\al_2+c\al_3}-1}-\log(e^{\al_3/\al_2+c \al_3}-1)+\log(\a_4)+\log\bfrac{e^{4\a_1}-1}{e^{4\a_3}-1}\\
\noalign{We need to be a little careful here. Now $\a_3/\a_2\leq 10$
  and if  $\a_3/\a_2\geq 9$ then $\a_3\geq 0.54$ and then $\a_i\leq
  0.46-3\times.06=0.28$ for $i\neq 3$. We bound
  $-\frac{1}{b}\cdot\frac {\partial b}{\partial \al_i}$ for both
  possibilities. Continuing we get}
\frac{\a_3}{\a_2}\geq 9:&\frac{1}{b}\cdot\frac {\partial b}{\partial \al_2}>-1.12-10.01-12.4-2.82-3.62=-29.97,\\
\frac{\a_3}{\a_2}\leq 9:&\frac{1}{b}\cdot\frac {\partial b}{\partial \al_2}>-2.4-9.01-11.4-2.82-3.62,\\
\frac{1}{b}\cdot\frac {\partial b}{\partial \al_2}\leq&4\brac{\a_3+\frac{\al_3}{e^{4\al_2}-1}}-\log(\a_2)+\log\bfrac{e^{4\a_1}-1}{e^{4\a_3}-1}-\log(e^{\al_3/\al_2+c \al_3}-1)\\
<&2.4\times 3.69+2.82+3.62+0.91.
\end{align*}
Finally, we find that 
\begin{multline}
\frac {\partial b}{\partial \al_3}=b(c,\al_0,\al_1,\al_2,\al_3)\times\\
c\left(-\al_0+\al_4\frac{e^{c \al_3}}{e^{c \al_3}-1}\right)+\frac{1+c \al_2}{e^{\al_3/\al_2}e^{c \al_3}-1}+\log(\al_4)-\log(\al_3)+\log\bfrac{e^{\al_2 c}-1}{e^{\al_3 c}-1}
\end{multline}
and so for $c=4$
\begin{align*}
\frac1b\cdot\frac {\partial b}{\partial \al_3}&\geq-4\a_0+\log(\a_4) +\log(e^{4\al_2}-1)-\log(e^{4\al_3}-1)>-2.4-2.82-3.62,\\
\frac1b\cdot\frac {\partial b}{\partial \al_3}&\leq 4\al_4\frac{e^{4\al_3}}{e^{4\al_3}-1}+\frac{1+4\al_2}{e^{\al_3/\al_2}e^{4 \al_3}-1}-\log(\a_3)+\log\bfrac{e^{4\al_2}-1}{e^{4\al_3}-1}\\
&<2.4\times 4.69+8.40+2.82+3.62.
\end{align*}
\newcommand{\bal}{\boldsymbol\alpha}
We see that $|\frac1b\cdot\frac{\partial b}{\partial \al_i}|<30$  for all $0\leq i\leq 3$. 
Thus, if we know that $b(c,\al_0,\al_1,\al_2,\al_3)\leq B$ for some $B$, this means that we can bound $b(4,\al_0,\al_1,\al_2,\al_3)<\rho$ by checking that $b(4,\al_0,\al_1,\al_2,\al_3)<\rho-\ep$ on a grid with step-size  $\delta\leq \ep/(2\cdot B\cdot 30)$.  

The C++ program in Appendix \ref{code} checks that $b(4,\al_0,\al_1,\al_2,\al_3)<.949$ on a grid with step-size $\delta=.0008$ (it completes in around an hour or less on a standard desktop computer, and is available for download from the authors' websites).   Suppose now that $B\geq 1$ is the supremum of $b(4,\al_0,\al_1,\al_2,\al_3)$ in the region of interest.  For $\ep=60 \delta B=0.048B$, we must have at some $\delta$-grid point that $b(4,\al_0,\al_1,\al_2,\al_3)\geq B-\ep=.962B\geq .962$.  This contradicts the computer-assisted bound of $<.949$ on the grid, completing the proof of Theorem \ref{t.l4}.\qed

\newpage
\appendix
\section{C++ code to check function bound}
\label{code}
\begin{verbatim}
#include <iostream>
#include <math.h>
#include <stdlib.h>
using namespace std;
int main(int argc, char* argv[]){
  double delta=.0008;        //step size                                                                                        
  double maxIndSet=.6;      //no independent sets larger than this fraction                                                    
  double minClass=.06;      //all color classes larger than this fraction                                                        
  double val=0;
  double maxval=0;
  double maxa0,maxa1,maxa2,maxa3; //to record the coordinates of max value                                                        
  maxa0=maxa1=maxa2=maxa3=0;
  double A23,A,B,C;           //For precomputing parts of the function                                                         
  double c=4;
  for (double a3=minClass; a3 + 4*minClass<1; a3+=delta){
    B=exp(c*a3)-1;
    for (double a2=minClass; a3 + a2 + 3*minClass<1; a2+=delta){
      A23=1/(pow(a2,a2)*pow(a3,a3)) * exp(-c/2)
                 * pow(exp(c*a2)-1,a3) * pow(1-exp(-a3/a2)*exp(-c*a3),a2);
      for (double a1=minClass;
           a3+a1<maxIndSet && a3 + a2 + a1 + 2*minClass<1;
           a1+=delta){
        A=A23/pow(a1,a1)* pow(exp(c*a1)-1,a2);
        for (double a0=max(max(minClass,.4-a2-a3),.4-a1-a3);
             a2+a0<maxIndSet && a3+a0<maxIndSet
              && a3 + a2 + a1 + a0 + minClass<1;
             a0+=delta){
          double a4=1-a0-a1-a2-a3;
          C=exp(c*a0);
          val=1/pow(a0,a0) * A * pow(B*C/a4,a4)* pow(C-1,a1);
          if (val>maxval){
            maxval=val; 
            maxa0=a0; maxa1=a1; maxa2=a2; maxa3=a3;
          }
        }
      }
    }
  }
  cout << "Max is "<<maxval<<", obtained at ("
       <<maxa0<<","<<maxa1<<","<<maxa2<<","<<maxa3<<","
       <<1-maxa0-maxa1-maxa2-maxa3<<")"<<endl;
}
\end{verbatim}
\newpage
program output:
\begin{verbatim}
$./bound
Max is 0.948754, obtained at (0.2904,0.2568,0.1704,0.1632,0.1192)
\end{verbatim}
\old{
\begin{algorithmic}
 \Procedure{Bound-$b$}{$\rho$}
  \State $M\gets 0$.
  \For {($x_0=.06;\quad x_0<.6;\quad x_0\gets x_0+\Delta_0$)}
   \State $\Delta_0\gets 1$
   \For {($x_1=.06;\quad x_0+x_1<.82;\quad x_1\gets x_1+\Delta_1$)}
    \State $\Delta_1\gets 1$
    \For {($x_2=.06;\quad x_0+x_1+x_2<.88;\quad x_2\gets x_2+\Delta_2$)}  
     \State $\Delta_2\gets 1$
     \For {($x_3=.06;\quad .4<x_0+x_1+x_2+x_3<.94;\quad x_3\gets x_3+\Delta_3$)}  
      \State $M\gets \max(M,b(c,x_0,x_1,x_2,x_3))$
      \State $\Delta_3 \gets ((\rho-b(4,x_0,x_1,x_2,x_3))+(\rho-.95))/(4\cdot 250)$
      \State $\Delta_2 \gets \min(\Delta_2,\Delta_3)$
     \EndFor
     \State $\Delta_1 \gets \min(\Delta_1,\Delta_2)$
    \EndFor
   \State $\Delta_0 \gets \min(\Delta_0,\Delta_1)$
   \EndFor
  \EndFor
 \If{$M<.95$}
  \State \Return{``$b$ bounded by $\rho$''}
 \Else{}
  \State \Return{``algorithm assumptions failed''}
  \EndIf
 \EndProcedure
\end{algorithmic}
(Note that if $b$ was not bounded by $\rho$, the algorithm might well never terminate.)
}

\end{document}